\documentclass[12pt,reqno]{amsart}

\usepackage{amssymb,amsthm}
\usepackage{amsfonts,cmtiup,comment,stmaryrd}

\usepackage{mathrsfs,cmtiup}

\makeatletter
\def\blfootnote{\xdef\@thefnmark{}\@footnotetext}
\makeatother

\newtheorem{theorem}{Theorem}[section]
\newtheorem{lemma}[theorem]{Lemma}
\newtheorem{prop}[theorem]{Proposition}
\newtheorem{corollary}[theorem]{Corollary}

\theoremstyle{definition}

\newtheorem{remark}[theorem]{Remark}
\newtheorem*{definition*}{Definition}
\newcommand{\ed}{\end{document}}
\newcommand{\F}{\mathbb{F}}
\newcommand{\f}{\varphi}
\renewcommand{\geq}{\geqslant}
\renewcommand{\leq}{\leqslant}

\let\leq=\leqslant
\let\geq=\geqslant
\numberwithin{equation}{section}

\begin{document}
\title{Finite groups with a soluble group\\ of coprime automorphisms whose fixed points\\ have bounded Engel sinks}

\author{E. I. Khukhro}
\address{Charlotte Scott Research Centre for Algebra, University of Lincoln, U.K., and \newline \indent  Sobolev Institute of Mathematics, Novosibirsk, 630090, Russia}
\email{khukhro@yahoo.co.uk}

\author{P. Shumyatsky}

\address{Department of Mathematics, University of Brasilia, DF~70910-900, Brazil}
\email{pavel@unb.br}


\begin{abstract}
Suppose that a finite group $G$ admits a soluble group of coprime automorphisms $A$. We prove that if, for some positive integer $m$, every element of the centralizer $C_G(A )$ has a left Engel sink of cardinality at most $m$ (or a right Engel sink of cardinality at most $m$), then $G$ has a subgroup of $(|A|,m)$-bounded index which has Fitting height at most $2\alpha (A)+2$, where $\alpha (A)$ is the composition length of $A$. We also prove that if, for some positive integer $r$, every element of the centralizer $C_G(A )$ has a left Engel sink of rank at most $r$ (or a right Engel sink of rank at most $r$), then $G$ has a subgroup of $(|A|,r)$-bounded index which has Fitting height at most $4^{\alpha (A)}+4\alpha (A)+3$. Here, a left Engel sink of an element $g$ of a group $G$ is a set ${\mathscr E}(g)$ such that for every $x\in G$ all sufficiently long commutators $[...[[x,g],g],\dots ,g]$ belong to ${\mathscr E}(g)$.  (Thus, $g$ is a left Engel element precisely when we can choose ${\mathscr E}(g)=\{ 1\}$.) A right Engel sink of an element $g$ of a group $G$ is a set ${\mathscr R}(g)$ such that for every $x\in G$ all sufficiently long commutators $[...[[g,x],x],\dots ,x]$ belong to ${\mathscr R}(g)$.  (Thus, $g$ is a right Engel element precisely when we can choose ${\mathscr R}(g)=\{ 1\}$.)

\end{abstract}

\dedicatory{To Victor Danilovich Mazurov on his 80th birthday}
\maketitle

\section{Introduction}
\baselineskip=16pt

Studying finite groups in relation to fixed points of their automorphisms is one of the mainstays of research in group theory contributing to  structural results and driving development of new methods. One of the best-known examples is J.~Thompson's theorem~\cite{tho} on the nilpotency of a finite group with a fixed-point-free automorphism of prime order. Using the classification of finite simple groups, P.~Rowley~\cite{row} showed that a finite group with a fixed-point-free automorphism of any order is soluble.

Studying groups with `almost fixed-point-free' automorphisms refers to a more general situation, when a group of automorphisms  $A\leq \operatorname{Aut}G$ may have non-trivial fixed-point subgroup $C_G(A)$, which is assumed to be `small' in one way or another. The usual goal is to prove that the group is `almost' as nice  as in the fixed-point-free case, where `almost' means modulo `pieces' commensurable in one way or another with  the fixed-point subgroup $C_G(A)$.  J.~Thompson~\cite{tho64} proved that the Fitting height of a finite soluble group $G$ with a soluble group of automorphisms $A$ of coprime order is bounded in terms of the Fitting height of $C_G(A)$ and the composition length $\alpha(A)$ of~$A$. (Recall that the Fitting height is the length of a shortest normal series with nilpotent factors.)   Thompson's paper inspired numerous subsequent papers dedicated to similar problems. Some of the best results in this area were obtained by A.~Turull; see his survey~\cite{tur}. Using Turull's results B.~Hartley and I.~M.~Isaacs~\cite{ha-is} proved that if a finite soluble group $G$ admits a soluble  group of automorphisms $A$ of coprime order, then the index $|G:F_{2\alpha (A)+1}(G)|$ is bounded in terms of $|C_G(A)|$ and $|A|$, where $F_{2\alpha (A)+1}(G)$ is the corresponding term of the Fitting series.

It is worth mentioning that when a group of automorphisms $A\leq \operatorname{Aut}G$ is not nilpotent, the coprimeness condition is essential for obtaining bounds of the Fitting height  even in the fixed-point-free case  $C_G(A)=1$, as shown by examples of S.~D.~Bell and B.~Hartley~\cite{be-ha}. But for a nilpotent group of automorphisms $A\leq \operatorname{Aut}G$ of a finite soluble group $G$ such that $C_G(A)=1$ the Fitting height of $G$ is bounded in terms of $\alpha(A)$ even without the coprimeness condition; this is a special case of a theorem of E.~C.~Dade~\cite{dade} who proved that the Fitting height of a finite soluble group $G$ is bounded in terms of $\alpha(H)$, where $H$ is a Carter subgroup of $G$.

While the aforementioned results dealt with restrictions on the order of $C_G(A)$ or its Fitting height, other `smallness' conditions on $C_G(A)$ had been successfully used.   E.~I.~Khukhro and V.~D.~Mazurov \cite{khu-maz06,khu-maz07,khu-maz05,khu-maz06semr} considered finite groups $G$ (not assumed to be soluble) having a group of automorphisms~$A$ with   fixed-point subgroup  $C_G(A)$ of given rank $r$. Here the rank of a group $H$ is the minimum positive integer $r$ such that every subgroup of $H$ can be generated by $r$ elements. One of the results in \cite{khu-maz07,khu-maz06semr}, when $A$ is soluble and has coprime order, is a bound for the rank of $G/ F_{4^{\alpha(A)}-1}$ in terms of $r$ and $|A|$.

In the present paper  we consider finite groups $G$ with a soluble group of automorphisms $A$ of coprime order such that all elements of $C_G(A)$ satisfy Engel-type conditions, which are expressed in terms of so-called Engel sinks (see the definition below). Given the maximum cardinality $m$ of these sinks, we prove that  the index $|G:F_{2\alpha (A)+2}(G)|$ is bounded in terms of $m$ and $|A|$  (Theorem~\ref{t-o}).  Given the maximum rank  $r$ of these sinks, we prove that the rank of  $G/F_{4^{\alpha (A)}+4\alpha (A)+3}(G)$ is bounded in terms of $r$ and $|A|$  (Corollary~\ref{c-1}), and if $G$ is in addition soluble, then the index   $|G:F_{4^{\alpha (A)}+4\alpha (A)+3}(G)|$ is bounded in terms of $r$ and $|A|$  (Theorem~\ref{t-r2}). The results obtained are in line with the studies of `almost fixed-point-free' automorphisms: if all elements of $C_G(A)$ are left or right Engel (that is, their left or right Engel sinks are trivial), then  $C_G(A)\leq F(G)$ by R.~Baer's theorem  \cite[12.3.7]{rob}, so that $A$ induces a fixed-point-free group of automorphism on $G/F(G)$, and therefore $G$ is soluble by a theorem of Y. M. Wang and Z. M. Chen \cite{wan-che} based on the classification of finite simple groups, and the Fitting height of $G$ is bounded in terms of $\alpha(A)$ by J.~Thompson's theorem~\cite{tho64} (with the improved bound $2\alpha(A)$ of A.~Turull~\cite{tur84}).

In order to state the precise results, we need the corresponding definitions.  We use the left-normed simple commutator notation
$[a_1,a_2,a_3,\dots ,a_r]:=[...[[a_1,a_2],a_3],\dots ,a_r]$
and the abbreviation $[a,\,{}_kb]:=[a,b,b,\dots, b]$ where $b$ is repeated $k$ times.

\begin{definition*} \label{dl}
 A \textit{left Engel sink} of an element $g$ of a group $G$ is a set ${\mathscr E}(g)$ such that for every $x\in G$ all sufficiently long commutators $[x,g,g,\dots ,g]$ belong to ${\mathscr E}(g)$, that is, for every $x\in G$ there is a positive integer $l(x,g)$ such that
 $[x,\,{}_{l}g]\in {\mathscr E}(g)$ for all $l\geq l(x,g).
 $
 \end{definition*}
 \noindent (Thus, $g$ is a left Engel element precisely when we can choose ${\mathscr E}(g)=\{ 1\}$, and $G$ is an Engel group when we can choose ${\mathscr E}(g)=\{ 1\}$ for all $g\in G$.)

 \begin{definition*} \label{dr}
 A \textit{right Engel sink} of an element $g$ of a group $G$ is a set ${\mathscr R}(g)$ such that for every $x\in G$ all sufficiently long commutators $[g,x,x,\dots ,x]$ belong to ${\mathscr R}(g)$, that is, for every $x\in G$ there is a positive integer $r(x,g)$ such that
 $[x,\,{}_{r}g]\in {\mathscr R}(g)$ for all $r\geq r(x,g).
 $
 \end{definition*}
 \noindent (Thus, $g$ is a right Engel element precisely when we can choose ${\mathscr R}(g)=\{ 1\}$, and $G$ is an Engel group when we can choose ${\mathscr R}(g)=\{ 1\}$ for all $g\in G$.)

 When $G$ is a finite group, every element   has the \emph{smallest} left Engel sink, since the intersection of two left Engel sinks ${\mathscr E}'(g)$ and ${\mathscr E}''(g)$ is again a left Engel sink of $g$. Similarly,
every element  $g\in G$  has the \emph{smallest} right Engel sink. In this paper we shall always use the notation ${\mathscr E}(g)$ and ${\mathscr R}(g)$  to denote the smallest left and right Engel sinks of $g$, respectively, thus eliminating the ambiguity of this notation in the above definitions. In cases where we need to consider the left or right Engel sink constructed with respect to a subgroup $H$ containing~$g$, we write ${\mathscr E}_H(g)$ or ${\mathscr R}_H(g)$, respectively.

 \begin{definition*} \label{drank}
 We say for short that an element $g$ of a group $G$ has a \emph{left (or right) Engel sink of rank at most $r$} if there is a subgroup of rank at most $r$ that is a left (respectively, right) Engel sink of~$g$. This is equivalent to the subgroup $\langle {\mathscr E}(g)\rangle$ (respectively,  $\langle {\mathscr R}(g)\rangle$) having rank at most $r$.
 \end{definition*}

 Recall that the Fitting series starts with the Fitting subgroup $F_1(G)=F(G)$, which is the largest normal nilpotent subgroup, and then by induction, $F_{k+1}(G)$ is the inverse image of $F(G/F_k(G))$. The Fitting height  of a finite soluble group $G$ is the least $h$ such that $F_h(G)=G$. Throughout the paper we write, say, ``$(a,b,\dots )$-bounded'' to abbreviate ``bounded above in terms of $a, b,\dots $ only''. We can now state the main results of the paper, the first of which relates to the cardinalities of Engel sinks.

\begin{theorem}\label{t-o}
Let $G$ be a finite group admitting a soluble group of automorphisms $A$ of order coprime to $|G|$. Let $m$ be a positive integer such that either
\begin{itemize}
  \item[\rm (a)] every element $g\in C_G(A)$ has a left Engel sink ${\mathscr E}_G(g)$  of cardinality at most $m$, or
  \item[\rm (b)] every element $g\in C_G(A)$ has a right Engel sink ${\mathscr R}_G(g)$ of cardinality at most $m$.
\end{itemize}
Then the index of the soluble radical $|G:S(G)|$ is bounded in terms of $m$, and the index  $|G:F_{2\alpha (A)+2}(G)|$ is bounded in terms of $|A|$ and $m$, where $\alpha (A)$ is the composition length of $A$.
\end{theorem}

Here we used the subscript $G$ in  ${\mathscr E}_G(g)$ and  ${\mathscr R}_G(g)$ to emphasize that these sinks are considered in the whole group $G$.

 It is not yet clear if one can exclude the dependence on $|A|$  in the conclusion of Theorem~\ref{t-o} for the index $|G:F_{2\alpha (A)+2}(G)|$. This is known to be possible when $|A|$ is a prime: in~\cite{khu-shu204} we considered a group $G$ satisfying the conditions of Theorem~\ref{t-o} when $A$ is of prime order and proved that then in case (a) the group $G$ has a metanilpotent normal subgroup of $m$-bounded index, and in case (b) the group $G$ has a nilpotent normal subgroup of $m$-bounded index. The proof of Theorem~\ref{t-o} relies on several useful lemmas in~\cite{khu-shu204}. Furthermore, the fact that $|G:S(G)|$ is $m$-bounded has already been established in \cite[Proposition~3.2]{khu-shu204}, so the new information in Theorem~\ref{t-o} relates to the index $|G:F_{2\alpha (A)+2}(G)|$.

 The other two main results relate to the rank of Engel sinks.

\begin{theorem}\label{t-r1}
Let $G$ be a finite group admitting a soluble group of automorphisms $A$ of order coprime to $|G|$. Let $r$ be a positive integer such that
every element $g\in C_G(A)$ has a left or right Engel sink  $\langle {\mathscr E}_G(g)\rangle$ or  $\langle {\mathscr R}_G(g)\rangle$  of rank at most $r$.
Then the rank of the quotient $G/S(G)$ by the soluble radical  is bounded in terms of $r$.
\end{theorem}

In Theorem~\ref{t-r1},  ``left or right'' is applied individually to elements of $C_G(A)$, so that it may be different, left or right, Engel  sinks for different  elements of $C_G(A)$.

\begin{theorem}\label{t-r2}
Let $G$ be a soluble finite group admitting a soluble group of automorphisms $A$ of order coprime to $|G|$. Let $r$ be a positive integer such that either
\begin{itemize}
  \item[\rm (a)] every element $g\in C_G(A)$ has a left Engel sink $\langle {\mathscr E}_G(g)\rangle$  of rank at most $r$, or
  \item[\rm (b)] every element $g\in C_G(A)$ has a right Engel sink   $\langle {\mathscr R}_G(g)\rangle$  of rank at most $r$.
\end{itemize}
Then the index  $|G:F_{4^{\alpha (A)}+4\alpha (A)+3}(G)|$  is bounded in terms of $|A|$ and $r$, where $\alpha (A)$ is the composition length of $A$.
\end{theorem}

Combination of Theorems~\ref{t-r1} and \ref{t-r2} yields the following.

\begin{corollary}\label{c-1}
  Let $G$ be a finite group admitting a soluble group of automorphisms $A$ of order coprime to $|G|$. Let $r$ be a positive integer such that either
\begin{itemize}
  \item[\rm (a)] every element $g\in C_G(A)$ has a left Engel sink $\langle {\mathscr E}_G(g)\rangle$  of rank at most $r$, or
  \item[\rm (b)] every element $g\in C_G(A)$ has a right Engel sink   $\langle {\mathscr R}_G(g)\rangle$  of rank at most $r$.
\end{itemize}
Then the rank of $G/F_{4^{\alpha (A)}+4\alpha (A)+3}(G)$  is bounded in terms of $|A|$ and $r$, where $\alpha (A)$ is the composition length of $A$.
\end{corollary}

Note that while it is well-known that the inverse of a right Engel element is left Engel, there is no such a straightforward connection between left and right Engel sinks, and parts (b) in Theorems~\ref{t-o} and~\ref{t-r2} and Corollary~\ref{c-1} are not consequences of parts (a). Rather, the proofs of these parts are similar and are mostly conducted simultaneously.

Other results concerning automorphisms whose fixed points satisfy restrictions on their Engel sinks were  obtained in \cite{acc-khu-shu, acc-shu-sil18, acc-shu-sil19, acc-sil18, acc-sil20, khu-shu202}. Generalizations of Engel conditions for finite, profinite, and compact groups  using the concept of Engel sinks were  considered in  \cite{khu-shu16,
khu-shu18,
khu-shu18a,
khu-shu19,
khu-shu20,
khu-shu-tra}. The natural aim of such studies is proving that groups (in certain classes) that have `small' Engel sinks have properties close to the nice properties of Engel groups (in which the sinks are trivial). As examples of such nice properties we mention that finite Engel groups are nilpotent by M.~Zorn's theorem \cite[12.3.4]{rob}, while
 J.~S.~Wilson and E.~I.~Zelmanov \cite{wi-ze} proved that profinite Engel groups are locally nilpotent. At the same time, E.~S.~Golod's examples \cite{gol} show that Engel groups in general may not be locally nilpotent.

For finite groups we proved in  \cite{khu-shu18} and  \cite{khu-shu19} the following quantitative results, which are used in this paper.

\begin{theorem}[{\cite[Theorem~3.1]{khu-shu18}}]\label{t-f}
Let $G$ be a finite group, and $m$  a positive integer. Suppose that every element $g\in  G$ has a left Engel sink ${\mathscr E}(g)$ of cardinality  at most~$m$.
 Then $G$ has a normal subgroup $N$ of order bounded in terms of $m$ such that $G/N$ is nilpotent.
\end{theorem}

\begin{theorem}[{\cite[Theorem~3.1]{khu-shu19}}]\label{t-fr}
Let $G$ be a finite group, and $m$  a positive integer. Suppose that every element $g\in  G$ has a right Engel sink ${\mathscr R}(g)$ of cardinality  at most~$m$.
 Then $G$ has a normal subgroup $N$ of order bounded in terms of $m$ such that $G/N$ is nilpotent.
\end{theorem}

\section{Preliminaries}\label{s-p}

For a group $A$ acting by automorphisms on a group $B$ we use the usual notation for commutators $[b,a]=b^{-1}b^a$ and commutator subgroups $[B,A]=\langle [b,a]\mid b\in B,\;a\in A\rangle$, as well as for centralizers $C_B(A)=\{b\in B\mid b^a=b \text{ for all }a\in A\}$ and $C_A(B)=\{a\in A\mid b^a=b \text{ for all }b\in B\}$.
We usually denote by the same letter $\varphi $ the automorphism induced by an automorphism $\varphi  $ on the quotient by a normal $\varphi $-invariant subgroup.

 We say for short that an automorphism $\varphi$ of a finite group $G$ is a coprime automorphism if  the orders of $\varphi$ and $G$ are coprime: $(|G|,|\varphi|)=1$.  The following lemma collects several well-known properties of coprime automorphisms of finite groups.

\begin{lemma}\label{l-nakr}
Let $A$ be a group acting by coprime automorphisms on a finite group $G$.
\begin{itemize}
  \item[\rm (a)]  If $N$ is a normal $A$-invariant subgroup of~$G$, then the fixed points of $A$ in  the quotient $G/N$ are covered by the fixed points of $A$ in $G$, that is, $C_{G/N}(A) = C_G(A)N/N$.

  \item[\rm (b)] We have $[[G, A],A]=[G,A]$.

  \item[\rm (c)]   If $G$ is abelian, then $G=[G,A]\times C_G(A)$.
\end{itemize}
  \end{lemma}

We shall need another useful lemma on coprime actions.

  \begin{lemma}[{\cite[Lemma~2.4]{khu-shu204}}]\label{l33}
Let $V$ be an abelian  finite group,  $U$ a group of coprime
automorphisms of $V$, and $m$ a positive integer.  If $|[V,u]|\leq m$ for every $u\in U$, then the  order $|U|$ is bounded in terms of $m$.
\end{lemma}

Recall that the rank of a finite group is the minimum positive integer $r$ such that every subgroup of it can be generated by $r$ elements. The following result was obtained by Kov\'acs~\cite{kov} for
soluble groups, and extended independently by Guralnick \cite{gur} and Lucchini \cite{luc} using the classification of finite simple groups (improving a bound $2d$ of Longobardi and Maj \cite{lo-ma}).

\begin{lemma}\label{l-kov} If $d$ is the maximum of the ranks of
the Sylow subgroups of a~finite group, then the rank of this
group is at most~$d+1$.
\end{lemma}

The following lemma appeared independently and simultaneously in
the papers of Gorchakov~\cite{grc}, Merzlyakov~\cite{me}, and as
``P.~Hall's lemma" in the paper of Roseblade~\cite{rs}.

\begin{lemma}\label{l-gmh}
 Let $p$ be a~prime number. The rank of a~$p$-group of
automorphisms of an abelian
finite $p$-group of rank~$r$ is bounded in
terms of~$r$.
\end{lemma}

The next two lemmas are also well-known facts.

\begin{lemma}[{see, for example, \cite[Lemma~1.3]{khu-shu18a}}] \label{l-r-coprime}
 A~finite $p'$-group of linear
transformations of a vector space of dimension $n$ over a field of
characteristic $p$ has $n$-bounded rank.
\end{lemma}

\begin{lemma}\label{l-exp}
A group of rank $r$ and of exponent $e$ has $(r,e)$-bounded order.
\end{lemma}

(The proof of this lemma does not have to use the positive solution of the Restricted Burnside Problem, as it follows from the bounds for the orders of Sylow subgroups, which can be obtained by using powerful $p$-groups.)

We now prove a rank analogue of Lemma~\ref{l33}.

  \begin{lemma}\label{l33r}
Let $V$ be an elementary abelian  finite $p$-group,  $U$ a group of coprime
automorphisms of $V$, and $r$ a positive integer.  If the rank of $[V,u]$ is at most $r$ for every $u\in U$, then
the rank of $U$ is bounded in terms of $r$.
\end{lemma}

\begin{proof}
We regard $V$ as a vector space over $\F_p$ in additive notation; then the rank of a~subgroup is its dimension. 
  First suppose that $U$ is abelian. Pick $u_1\in U$ such that $[V,u_1]\ne 0$. By Lemma~\ref{l-nakr}(c), $V= [V,u_1]\oplus  C_V(u_1)$, and both summands are $U$-invariant, since $U$ is abelian. If $C_U([V,u_1])=1$, then the rank of $U$ is $r$-bounded by Lemma~\ref{l-r-coprime}.
  Otherwise pick $1\ne u_2\in C_U([V,u_1])$; then $V= [V,u_1] \oplus [V,u_2] \oplus C_V(\langle u_1,u_2\rangle )$. If $1\ne u_3\in C_U([V,u_1]\oplus [V,u_2])$, then $V= [V,u_1]\oplus [V,u_2]\oplus [V,u_3] \oplus C_V(\langle u_1,u_2,u_3\rangle )$, and so on. If $C_U([V,u_1]\oplus \dots \oplus [V,u_k])=1$ at some $r$-bounded step $k$, then $U$ has $r$-bounded rank by Lemma~\ref{l-r-coprime}. However, if there are too many steps, then
for the element $w=u_1u_2\cdots u_k$ with $k>r$ by Lemma~\ref{l-nakr}(b) we shall have $0\ne [V,u_i]= [[V,u_i],w]$ for every $i=1,2,\dots,k$, so that $[V,w] = [V,u_1]\oplus \dots \oplus [V,u_k]$ will have rank greater than $r$, a contradiction.

We now consider the general case.  If $Q$ is a Sylow $q$-subgroup of $U$, let $M$ be a maximal normal abelian subgroup of $Q$. By the above, the rank of  $M$ is $r$-bounded. Hence $Q$ has $r$-bounded rank by Lemma~\ref{l-gmh}, since $C_P(M)= M$ and $P/M$ embeds in the automorphism group of $M$.  Since the rank of a Sylow $q$-subgroup of $U$ is $r$-bounded for every prime divisor $q$ of $|U|$, the rank of $U$ is $r$-bounded by Lemma~\ref{l-kov}.
\end{proof}

  We now recall some elementary properties of Engel sinks.   The following lemma is a consequence of the properties of coprime actions and a formula in H.~Heineken's paper~\cite{hei} on a connection between left and right Engel commutators.

\begin{lemma}[{\cite[Lemma~3.2]{khu-shu19}}]\label{l-sink}
 If $V$ is an abelian subgroup of a finite group $G$, and $g\in G$ an element normalizing $V$  such that $(|V|,|g|)=1$, then
 $$[V,g]={\mathscr E}_{V\langle g\rangle}(g)={\mathscr R}_{V\langle g\rangle}(g).$$
\end{lemma}

\begin{remark}\label{r-inh}
Since the image of a group commutator $[a,{}_kb]$ under a homomorphism $\vartheta$ is the commutator $[a^{\vartheta},{}_kb^\vartheta]$  of the images, for any element $g$ of a group $G$ and any normal subgroup $N$ the image of a  left (right) Engel sink $ \mathscr E(g)$ (respectively, $\mathscr R(g)$) in $G/N$ is a   left (right) Engel sink of the image of $g$. Furthermore,
 for any  $A$-invariant section $M/N$, that is, a quotient  of an  $A$-invariant  subgroup $M$ by an  $A$-invariant subgroup $N$  normal in $M$, the centralizer of $A$ in $M/N$ is the image of the centralizer $C_M(A)$ by Lemma~\ref{l-nakr}. Therefore the hypotheses of Theorem~\ref{t-o} are inherited by any  $A$-invariant section. We shall freely use this fact without special references.
\end{remark}

\section{Sinks of bounded cardinality}\label{s-sol-card}

First we consider a soluble group $G$ of Fitting height 2 satisfying the hypotheses of Theorem~\ref{t-o}.

\begin{prop}\label{pr-h2}
Let $G$ be a finite group admitting a group of automorphisms $A$ of order coprime to $|G|$. Let $m$ be a positive integer such that either
\begin{itemize}
  \item[\rm (a)] every element $g\in C_G(A)$ has a left Engel sink ${\mathscr E}_G(g)$ of cardinality at most $m$, or
  \item[\rm (b)] every element $g\in C_G(A)$ has a right Engel sink ${\mathscr R}_G(g)$ of cardinality at most $m$.
\end{itemize}
If  $G=F_2(G)$, then the centralizer $C_{G/F(G)}(A)$ of $A$ in $G/F(G)$ has  $m$-bounded order.
\end{prop}

\begin{proof}
To lighten the notation, for the rest of the proof we write $F=F(G)$.
 By Gasch\"utz's theorem  \cite[Satz~III.4.2]{hup} the image of the Fitting subgroup $F$ in the quotient  of $G$ by the  Frattini subgroup $\Phi(F)$ of $F$ is the Fitting subgroup of the quotient $G/\Phi(F)$. Since $|C_{(G/\Phi(F))/(F/\Phi(F))}(A)|=|C_{G/F}(A)|$ by Lemma~\ref{l-nakr} and the hypothesis is inherited by the quotient $G/\Phi(F)$ by Remark~\ref{r-inh}, we can assume that  $\Phi(F)=1$. Then $F$ is a direct product of minimal normal subgroups of $G$ by  Gasch\"utz's theorem  \cite[Satz~III.4.5]{hup}, and therefore $F$ is a direct product of
elementary abelian $p$-groups for various~$p$.

Let $q$ be any prime dividing $|C_{G/F}(A)|$, and let $Q$ be a Sylow $q$-subgroup of $C_{G/F}(A)$. Consider any non-trivial element $g\in Q$, and let $\hat g$ be a pre-image of $g$ in $C_G(A)$, which exists by Lemma~\ref{l-nakr}. Since the image of the Sylow $q$-subgroup of $\langle\hat g\rangle$ in $G/F$ contains $g$, we can choose $\hat g$ to be also a $q$-element.

 Then $\hat g$ induces by conjugation a non-trivial coprime automorphism of the Hall $q'$-subgroup $F_{q'}$ of $F$. If we suppose the opposite, that $\hat g$ centralizes $F_{q'}$, then $\hat g$ centralizes all factors of a principal series of $G$. Indeed, such a series can be chosen to contain $F_{q'}$ and $F$. The element $\hat g$,  being a $q$-element, centralizes the principal  $q$-factors in $F/F_{q'}$; furthermore, $\hat g$ centralizes all factors of this series above $F$, since $G/F$ is nilpotent by hypothesis. As a result, then $\hat g$  would belong to $F$ by~\cite[5.2.9]{rob}, a~contradiction.

Defining   the action of $g\in Q$  on $F_{q'}$ as induced by the conjugation by an element $\hat g$ chosen as above, we obtain a well-defined faithful action of $Q$ by coprime automorphisms  on $F_{q'}$, because $\langle\hat g\rangle\cap F$ is contained in the Sylow $q$-subgroup of $F$ and therefore centralizes $F_{q'}$.
 We have $[F_{q'},\hat g]\subseteq \mathscr E(\hat g)\cap \mathscr R(\hat g)$ by Lemma~\ref{l-sink}. Hence, under any of the hypotheses (a) or (b), the subgroup  $[F_{q'},\hat g]=[F_{q'},  g]$ has $m$-bounded order for any $g\in Q$. Then $|Q|$ is also $m$-bounded by Lemma~\ref{l33}. In particular, the prime $q$ is bounded above in terms of $m$, so there are only $m$-boundedly many primes dividing  $|C_{G/F}(\varphi)|$. Since the order of $C_{G/F}(\varphi)$ is the product of the orders of its Sylow subgroups, each of which has $m$-bounded order,  we obtain that  $|C_{G/F}(\varphi)|$ is  bounded in terms of $m$.
\end{proof}

\begin{proof}[Proof of Theorem~\ref{t-o}]
Recall that  $G$ is a finite group admitting a soluble group of automorphisms $A$ of order coprime to $|G|$, and $m$ is a positive integer such that either
\begin{itemize}
  \item[\rm (a)] every element $g\in C_G(A)$ has a left Engel sink ${\mathscr E}_G(g)$ of cardinality at most $m$, or
  \item[\rm (b)] every element $g\in C_G(A)$ has a right Engel sink ${\mathscr R}_G(g)$ of cardinality at most $m$.
\end{itemize}
We already know from \cite[Proposition~3.2]{khu-shu204} that the index of the soluble radical $|G:S(G)|$ is bounded in terms of $m$. It remains to show that $|G:F_{2\alpha (A)+2}(G)|$ is bounded in terms of $|A|$ and $m$. For that, we obviously can assume that the group $G$ is soluble.

By Theorem~\ref{t-f} or \ref{t-fr} (in cases (a) or (b), respectively), we obtain that $C_G(A)$ has a subgroup of $m$-bounded order with nilpotent quotient. Therefore the Fitting height of $C_G(A)$ is $m$-bounded. Applying Thompson's theorem~\cite{tho64} (or Turull's~\cite{tur84} improvement with best possible bound), we obtain that the Fitting height of $G$  is $(\alpha(A),m)$-bounded.

We now apply Proposition~\ref{pr-h2} to every metanilpotent quotient $F_{i+2}(G)/F_i(G)$ for $i=0,1,2,\dots$, where $F_0(G)=1$. By this proposition, $|C_{F_{i+2}(G)/F_{i+1}(G)}(A)|$ is $m$-bounded, for every  $i=0,1,2,\dots$. Since the Fitting height of $G$ is  $(\alpha(A),m)$-bounded, in view of Lemma~\ref{l-nakr} we obtain that $|C_{G/F(G)}(A)|$ is $(\alpha(A),m)$-bounded.

By the Hartley--Isaacs theorem \cite[Theorem~A]{ha-is} applied to $G/F(G)$ and its group of automorphisms induced by $A$, it follows that the $(2\alpha (A)+1)$-th Fitting subgroup of $G/F(G)$ has $(|A|,m)$-bounded index, and therefore the index  $|G:F_{2\alpha (A)+2}(G)|$ is also $(|A|,m)$-bounded.
\end{proof}

\section{Bounding the rank of the quotient by the soluble radical}\label{s-red}

A bound for the index of the soluble radical in Theorem~\ref{t-o} has  been obtained   in~\cite[Proposition~3.2]{khu-shu204} as a corollary of a recent result of R.~M.~Guralnick and G.~Tracey~\cite{gur-tra}. The proof of  Theorem~\ref{t-r1}, where we obtain a bound for the rank of the quotient by the  soluble radical, is similar. First we quote one of the  results in~\cite{gur-tra}.

\begin{theorem}[{see \cite[Theorem~1.6]{gur-tra}}]\label{t-g-t}
   Let $\tau\in  \operatorname{Aut}G$ be an involutive automorphism of a finite group $G$, and let
$J(\tau)=\{g\in G\mid \text{$g$ has odd order and } g^{\tau}=g^{-1}\}$.
Suppose that $[G, \tau ] = G$. Then $G = \langle J(\tau )\rangle$.
\end{theorem}

We derive from Theorem~\ref{t-g-t} a proposition about groups of coprime automorphisms whose fixed points have Engel sinks of bounded rank in greater generality than Theorem~\ref{t-r1}, as it may find applications in other studies. 

\begin{prop}\label{pr}
Let $G$ be a finite group admitting a group of coprime automorphisms~$H$.
 Suppose that $r$ is a positive integer such that every $2$-element $t\in C_G(H)$ has a left or right Engel sink $\langle\mathscr E_G(t)\rangle$  or $\langle\mathscr R_G(t)\rangle$  of rank at most~$r$. Then the rank of the quotient $G/S(G)$ by the soluble radical is $r$-bounded.
\end{prop}

(In the statement,  ``left or right'' is applied individually to the $2$-elements of $C_G(H)$, so that it may be different, left or right, Engel  sinks for different  $2$-elements of $C_G(H)$.)

\begin{proof}
  Since the hypothesis is inherited by $G/S(G)$, we can assume that $S(G)=1$. Then the generalized Fitting subgroup
  \begin{equation}\label{e-gt}
    F^*(G)=T_1\times \dots \times T_n
  \end{equation}
  is a direct product of non-abelian finite simple groups  $T_i$, which are permuted by  $H$.
  Since  the centralizer of $F^*(G)$ is trivial, the group $G$ embeds into the automorphism group of $F^*(G)$, which is known to be of the form
    \begin{equation}\label{e-gta}
    \big(\operatorname{Aut} T_1\times \dots   \times  \operatorname{Aut}T_n\big)U,
  \end{equation}
  where $U$ is a subgroup of the symmetric group $S_n$. Since the ranks of the outer automorphism groups $ \operatorname{Aut}T_i/T_i$ are known to be at most~3 (due to the classification of finite simple groups; see, for example, \cite{atlas}), it is sufficient to obtain a bound in   terms of $r$ for the rank of  $F^*(G)$. Thus, we can simply assume that $G=F^*(G)$. 

  As proved by Y. M. Wang and Z. M. Chen \cite{wan-che} on the basis of the classification of finite simple groups, a finite group admitting a group of coprime automorphisms with fixed-point subgroup of odd order is soluble. Therefore $C_G(H)$ contains involutions. Moreover, we can choose an involution $\tau\in C_G(H)$ with non-trivial projections onto each of the factors $T_i$ in \eqref{e-gt}.   Indeed, let  $\{T_{i_1},\dots,T_{i_k}\}$ be one of the orbits of $H$ in its permutational action on the set $\{T_1,\dots,T_n\}$.   Clearly, any non-trivial element of the  product $T_{i_1}\times \cdots\times T_{i_k}$ centralized by $H$ must have non-trivial projections onto each of the factors $T_{i_j}$. There is an involution in the centralizer of $H$ in the product $T_{i_1}\times \cdots\times T_{i_k}$. We now take $\tau\in C_G(H)$ to be the product of these involutions over all orbits of $H$. Then $\tau\in C_G(H)$ is an involution with non-trivial projections onto each of the factors $T_{i}$ in \eqref{e-gt}. Since each $T_i$ is a non-abelian simple group, it follows that $G=F^*(G)=[G,\tau ]$.

  If $g\in J(\tau )$, then $[\langle g\rangle,\tau ]=\langle g\rangle$, since  $g^{\tau }=g^{-1}$ and $g$ has odd order. Then $\langle g\rangle\subseteq \mathscr R (\tau)\cap \mathscr E(\tau)$ by Lemma~\ref{l-sink}. Therefore, $J(\tau )\subseteq \langle\mathscr R (\tau)\rangle\cap \langle\mathscr E(\tau)\rangle$, so that $\langle J(\tau )\rangle$ has rank at most $r$. Since $[G,\tau ]=G$, we have $G=\langle J(\tau )\rangle$ by Theorem~\ref{t-g-t},  which completes the proof.
\end{proof}

In our situation, we obtain Theorem~\ref{t-r1} as the following obvious corollary.

\begin{corollary}
\label{l-srs}
Let $G$ be a finite group admitting a soluble group of automorphisms $A$ of order coprime to $|G|$. Let $r$ be a positive integer such that
every element $g\in C_G(A)$ has a left or right Engel sink  $\langle {\mathscr E}_G(g)\rangle$ or  $\langle {\mathscr R}_G(g)\rangle$  of rank at most $r$.
Then the rank of the quotient $G/S(G)$ by the soluble radical  is bounded in terms of $r$.
\end{corollary}

\section{Sinks of bounded rank in soluble case}\label{s-sol-rank}

In this section we prove Theorem~\ref{t-r2} about a soluble finite group $G$  admitting a soluble group of coprime automorphisms $A$ with  restrictions on the ranks of  Engel sinks of elements of $C_G(A)$. First we state one of the results of E.~I.~Khukhro and V.~D.~Mazurov   \cite{khu-maz07} about finite groups with automorphisms whose fixed point subgroups have bounded rank.

\begin{theorem}[{\cite[Theorem~5.2(c)]{khu-maz07}}]\label{t-km}
If a soluble finite group $G$ admits an automorphism $\f$ of prime order $p$ coprime to $|G|$ such that $C_G(\f )$ has rank $r$, then the order $|G/F_4(G)|$ is $(p,r)$-bounded.
\end{theorem}

While the paper \cite{khu-maz07} also contained a theorem  about a soluble group $A$ of coprime automorphisms, we derive here essentially the same result with a slightly better function.

\begin{corollary}\label{c-km}
If a soluble finite group $G$ admits a soluble group of coprime automorphisms $A$ such that $C_G(A)$ has rank $r$, then the order $|G/F_{4^{\alpha (A)}+2\alpha (A)-1}|$ is $(|A|,r)$-bounded.
\end{corollary}

\begin{proof}
Let $A_1$ be a normal subgroup of $A$ of prime index. The quotient $A/A_1$ induces an automorphism of prime order of the subgroup   $C_1=C_G(A_1)$ with centralizer  $C_{C_1}(A/A_1)=C_G(A)$ of rank $r$. By the Khukhro--Mazurov Theorem~\ref{t-km} the order $|C_1/F_4(C_1)|$ is bounded in terms of $|A/A_1|$ and $r$. One of the main results of J.~Thompson's paper~\cite{tho64} states that for a coprime automorphism $\f\in\operatorname{Aut}G$ of prime order we have the inclusion $F(C_G(\f ))\leq F_4(G)$. A repeated application of this theorem with respect to $A_1$ yields the inclusion $F_4(C_G(A_1))\leq F_{4\cdot 4^{\alpha(A)-1}}(G)=F_{4^{\alpha(A)}}(G)$ (see \cite[Corollary~5.6]{khu-maz07}. Then the order $|C_{G/F_{4^{\alpha(A)}}(G)}(A_1)|$ is bounded in terms of $|A/A_1|$ and $r$ by Lemma~\ref{l-nakr}(a). By the Hartley--Isaacs theorem \cite[Theorem~A]{ha-is} applied to $G/F_{4^{\alpha(A)}}(G)$ and its group of automorphisms induced by $A_1$, it follows that the index of the $(2(\alpha (A_1)+1)$-th Fitting subgroup of $G/F_{4^{\alpha(A)}}(G)$ is bounded in terms of  $|A_1|$ and $|C_{G/F_{4^{\alpha(A)}}(G)}(A_1)|$, and therefore, by the above,  in terms of $|A|$ and $r$. Taking into account the equation $\alpha (A_1)=\alpha (A)-1$, we obtain as a result that the index  $|G:F_{4^{\alpha(A)}+2\alpha (A)-1}(G)|$ is  $(|A|,r)$-bounded.
\end{proof}

We are now ready to prove Theorem~\ref{t-r2}. We use the standard notation $O_{p'}(G)$ for the largest normal $p'$-subgroup of $G$, and $O_{p',p}(G)$ for the inverse image of the largest normal $p$-subgroup of $G/O_{p'}(G)$.

\begin{proof}[Proof of Theorem~\ref{t-r2}] Recall that $G$ is a soluble finite group 
  admitting a soluble group of coprime automorphisms $A$ such that either
\begin{itemize}
  \item[\rm (a)] every element $g\in C_G(A)$ has a left Engel sink $\langle {\mathscr E}_G(g)\rangle$  of rank at most $r$, or
  \item[\rm (b)] every element $g\in C_G(A)$ has a right Engel sink   $\langle {\mathscr R}_G(g)\rangle$  of rank at most $r$.
\end{itemize}
We need to prove that the index  $|G:F_{4^{\alpha (A)}+4\alpha (A)+3}(G)|$  is bounded in terms of $|A|$ and $r$, where $\alpha (A)$ is the composition length of $A$.

  Fix a prime $p$ dividing $|G|$. The quotient $\bar G=G/O_{p',p}(G)$ acts faithfully
on the Frattini quotient $V$ of $O_{p',p}(G)/O_{p'}(G)$ (see, for example, \cite[Theorem~6.3.4]{gor}). Let $H$ be a Hall $p'$-subgroup of $C_G(A)$. For  $h\in H$, we have $[V,h]\subseteq \langle {\mathscr E}_{V\langle h\rangle}(h)\rangle\cap \langle {\mathscr R}_{V\langle h\rangle}(h)\rangle$ by Lemma~\ref{l-sink}. Since the subgroups $\langle {\mathscr E}_{V\langle h\rangle}(h)\rangle$ and $\langle {\mathscr R}_{V\langle h\rangle}(h)\rangle$ are images of subgroups of $\langle {\mathscr E}_{G}(h)\rangle$ and $\langle {\mathscr R}_{G}(h)\rangle$, respectively, it follows that under any of the above conditions (a) or (b)  the rank of $[V,h]$ is at most $r$ for any $h\in H$. Therefore the image $\bar H$ of $H$ in $\bar G$ has $r$-bounded rank by Lemma~\ref{l-r-coprime}.

Now let $t$ be any prime dividing $|\bar G|$ such that $t\ne p$. (Note that if there are no such primes $t$, then  $G=O_{p',p}(G)$.) The group $\bar{\bar{G}}=\bar G/O_{t',t}(\bar G)$ acts faithfully
on the Frattini quotient $W$ of $O_{t',t}(\bar G)/O_{t'}(\bar G)$.  Let $K$ be a Hall $t'$-subgroup of $C_G(A)$. For  $k\in K$, we have $[W,k]\subseteq \langle {\mathscr E}_{W\langle k\rangle}(k)\rangle\cap \langle {\mathscr R}_{W\langle k\rangle}(k)\rangle$ by Lemma~\ref{l-sink}. Since the subgroups $\langle {\mathscr E}_{W\langle k\rangle}(k)\rangle$ and $\langle {\mathscr R}_{W\langle k\rangle}(k)\rangle$ are images of subgroups of $\langle {\mathscr E}_{G}(k)\rangle$ and $\langle {\mathscr R}_{G}(k)\rangle$, respectively, it follows that under any of the above conditions (a) or (b)  the rank of $[W,k]$ is at most $r$ for any $k\in K$. Therefore the image $\bar{\bar K}$ of $K$ in $\bar{\bar G}$ has $r$-bounded rank by Lemma~\ref{l-r-coprime}.

Since $p\ne t$, for any prime $q$ dividing $|C_G(A)|$ there is a Sylow $q$-subgroup $Q$ of $C_G(A)$ that is contained either in the Hall $p'$-subgroup $H$ of  $C_G(A)$ or in the Hall $t'$-subgroup $K$ of $C_G(A)$. Hence the  image $\bar{\bar Q}$ of $Q$ in $\bar{\bar G}$ has $r$-bounded rank, since this was proved above for $H$ and $K$. By Lemma~\ref{l-kov} then the image of  $C_G(A)$ in $\bar{\bar G}$ also has $r$-bounded rank. Using Lemma~\ref{l-nakr}(a) we obtain that $C_{\bar{\bar G}}(A)$  has $r$-bounded rank.

By the Khukhro--Mazurov Corollary~\ref{c-km} the order $|\bar{\bar G}/F_{4^{\alpha (A)}+2\alpha (A)-1}(\bar{\bar G})|$ is bounded in terms of $|A|$ and the rank of   $C_{\bar{\bar G}}(A)$, and therefore in terms of $|A|$ and  $r$ by the above.

The quotient
\begin{equation}\label{e-inter}
\bar G\left/ \bigcap_{t\ne p}O_{t',t}(\bar G)\right.
\end{equation}
embeds into the direct product of the quotients $
\bar G\left/ O_{t',t}(\bar G)\right.
$ over the primes $t\ne p$. Each of these quotients (denoted by $\bar{\bar G}$ above for a particular $t$) is an extension of a normal subgroup of Fitting height at most $4^{\alpha (A)}+2\alpha (A)-1$ by a group of  $(|A|, r)$-bounded order. Therefore this direct product, as well as the quotient~\ref{e-inter}, is an extension of  a normal subgroup of Fitting height at most $4^{\alpha (A)}+2\alpha (A)-1$ by a group of  $(|A|, r)$-bounded exponent.

Since $O_p(\bar G)=1$, we have
$$
 \bigcap_{t\ne p}O_{t',t}(\bar G)= \bigcap_{q}O_{q',q}(\bar G)=F(\bar G).
$$
Thus, the quotient~\ref{e-inter} is equal to  $\bar G/F(\bar G)$ and is an extension of  a normal subgroup of Fitting height at most $4^{\alpha (A)}+2\alpha (A)-1$ by a group of $(|A|, r)$-bounded exponent. In other words, the quotient $\bar G=G/O_{p',p}(G)$ is an extension of  a normal subgroup of Fitting height at most $4^{\alpha (A)}+2\alpha (A)$ by a group of  $(|A|, r)$-bounded exponent for any prime $p$.  The same assertion is also true in the degenerate case, when there are no primes $t\ne p$ dividing $|G/O_{p',p}(G)|$, since then $G=O_{p',p}(G)$.

Since $F(G)= \bigcap_{p}O_{p',p}(G)$, the quotient
$$
G/F(G)= G\left/ \bigcap_{p}O_{p',p}(G)\right.
$$
embeds into the direct product of the quotients $
 G/ O_{p',p}( G)$  and therefore is also an extension of  a normal subgroup of Fitting height at most $4^{\alpha (A)}+2\alpha (A)$ by a group of  $(|A|, r)$-bounded exponent. Thus, the quotient  $\widetilde G=G/F_{4^{\alpha (A)}+2\alpha (A)+1}(G)$ has $(|A|, r)$-bounded exponent.

In a group of  $(|A|, r)$-bounded exponent, a subgroup of rank at most $r$ has  $(|A|, r)$-bounded order by Lemma~\ref{l-exp}. Therefore under condition (a) or (b) the left or right Engel sinks of elements of $C_{\widetilde G}(A)$ have $(|A|, r)$-bounded orders. Applying Theorem~\ref{t-o}
to $\widetilde G$ and its group of  automorphisms induced by $A$ we obtain that the order  $|\widetilde G/F_{2\alpha (A)+2}(\widetilde G)|$ is $(|A|, r)$-bounded. As a result, the order $|G/F_{4^{\alpha (A)}+4\alpha (A)+3}(G)|$ is  $(|A|, r)$-bounded.
\end{proof}

\end{document}